\documentclass[a4paper,10pt]{article}
\usepackage[utf8]{inputenc}
\usepackage{ucs}
\usepackage{amsmath}
\usepackage{amsfonts}
\usepackage{amssymb}
\usepackage{amsthm}
\usepackage[english]{babel}
\usepackage{fontenc}
\usepackage{graphicx}
\usepackage{geometry}
\usepackage{hyperref}
\usepackage{tikz}
\usepackage{tkz-euclide}
\usepackage[font=small,labelfont=sc]{caption}
\usepackage{fancyhdr}

\fancypagestyle{specialfooter}{%
  \fancyhf{}
   
  \fancyfoot[L]{\small\quad This work was supported by the UK Engineering and Physical Sciences Research Council (EPSRC) studentship award 1789662.}
}

\newtheorem{theorem}{Theorem}
\newtheorem*{theorem*}{Theorem}
\newtheorem{lemma}{Lemma}
\newtheorem*{lemma*}{Lemma}

\newtheorem{proposition}{Proposition}
\newtheorem*{proposition*}{Proposition}

\newtheorem*{conjecture*}{Conjecture}
\title{Bounds on the largest prime factor of a negative discriminant with one class per genus}
\author{John Armitage}

\begin{document}

\maketitle

%\begin{abstract}
% It was conjectured by Gauss that any negative discriminant with at least 32 genera has at least two classes of binary quadratic forms in each genus. We prove that such a discriminant cannot have a particularly large prime factor, by an argument similar to that of Baker's solution to Gauss' class number one problem.
%\end{abstract}

\section{Introduction}
\thispagestyle{specialfooter}

Euler's idoneal numbers are numbers $n$ such that if $m$ is an odd number properly and uniquely represented by the binary quadratic form $x^2+ny^2$, and $m$ is coprime to $n$, then $m$ is prime. Euler used these to produce large primes, and found 65 such numbers $n$, conjecturing that his list is complete.

By Gauss' theory of genera, those even negative discriminants $-4n$ which have one class of binary quadratic form in each genus yield idoneal numbers $n$, and the opposite implication was established by Grube \cite{Grube}. Euler's form $x^2+ny^2$ is then the principal form of the discriminant $-4n$. Gauss conjectured that if a negative discriminant has at least $32$ genera, there are at least two classes of binary quadratic forms in each genus. Their finiteness was proved by Heilbronn and Linfoot \cite{Heilbronn}, and Weinberger \cite{Weinberger} proved by means of Tatuzawa's theorem that only one further such fundamental discriminant may exist, and that no more exist under the Generalized Riemann Hypothesis, and it is further known that if $8\Vert d$, and $d$ is a fundamental discriminant with one class per genus, then $4d$ will also have one class per genus, and no further discriminants. A useful survey with a historical overview may be found in \cite{Kani}. We make use of a computation performed by the author implementing Swift's tests to give a lower bound of $9.8\cdot10^{18}$ for any further discriminant with one class per genus, described in section 2.

We consider here a special case of the one class per genus problem, where the discriminants under consideration have a large prime factor. The main argument proceeds in a similar manner to Baker's solutions of the small class number problems (see for example \cite{Baker}) considering an expansion in terms of the Epstein zeta functions associated to binary quadratic forms of products of $L$-functions, and reducing this to a linear form in logarithms. In our case, the height of the coefficients in the linear form in logarithms remains small as the minima of the binary quadratic forms are all divisors of the discriminant -- unlike both the general case of the one class per genus problem, and in general for discriminants, and the large prime factor ensures that the minima are small in comparison to the discriminant. We prove the following,

\begin{theorem}
 Suppose that $d$ is a negative fundamental discriminant with at least $32$ genera, and let $P$ its largest prime factor. Then $d$ has at least two classes of binary quadratic forms in each genus if
 \begin{equation*}
  P\geq 5\cdot10^{15}\lvert d\rvert^{1/2}(\log\lvert d\rvert)^5\log\log \lvert d\rvert.
 \end{equation*}
\end{theorem}
Watkins \cite{watkinsspectral} indicates that the spectral methods employed for his proof of the class number one problem would yield the impossibility of a prime factor $P\gg \lvert d\rvert^{1-\alpha}$, for some positive $\alpha$, though these methods too are restricted to the situation of small minima.

For a binary quadratic form $f(x,y)=ax^2+bxy+cy^2$ we use the notation $(a,b,c)$, and $a$ is its minimum if $f$ is reduced. By the theory of genera, if there is one class of binary quadratic forms per genus, then all forms are ambiguous, i.e. their square under Gaussian composition is the principal form, and ambiguous forms which are reduced have the form $(a,0,c)$, $(a,a,c)$, or $(a,b,a)$.

\section{Sieving small discriminants}

We describe here the algorithm used to eliminate negative fundamental discriminants with one class per genus. It is an algorithm to combine linear congruences, incorporating a method of bit--packing whereby at each iteration one 32--bit OR operation applies a linear congruence condition simultaneously to 32 potential discriminants.

\begin{proposition}
 There are no further negative discriminants with one class per genus satisfying $\lvert d\rvert\leq 9.8\cdot10^{18}$.
\end{proposition}
\begin{proof}
The main sieving test used in the algorithm is to test whether a negative discriminant is congruent to a non--zero quadratic residue $\beta=x^2$ modulo a prime $p<\sqrt{\lvert d\rvert}/2$ -- in this case, for one of the possibilities of $x$, $d= x^2-4kp$ for some $k$, which gives a reduced non--ambiguous form $(p,x,k)$ when $k>p$. However, when $k=p$, $p\geq\sqrt{\lvert d\rvert}/2$. As we sieve with primes up to $p_{169}$, supposing that $d$ is sieved by some prime, if $\lvert d\rvert\geq10^7$, then we obtain a non--ambiguous form. The range $\lvert d\rvert\leq10^7$ may be tested by a direct application of Swift's tests to each $d$.

Firstly two products $P_1, P_2$ of primes are taken, where $P_1 = 3\cdot5\cdot\cdots\cdot19$, and $P_2 = 23\cdot29\cdot\cdots\cdot47$. For each number from zero to $P_i$, we eliminate those values which are congruent to the negative of a non--zero quadratic residue modulo any prime dividing $P_i$ -- the surviving values are stored in an array for each $P_i$. 

An array of 32--bit bit--vectors (each vector stored as an unsigned integer) of length $q$, indexed by $a$, is prepared for each prime $q$, from the $16$'th to the $169$'th prime, where the $k$'th bit of the unsigned integer is set to $1$  if $a+kP_1P_2$ is the negative of a non--zero quadratic residue modulo $q$, and set to zero otherwise.

The program has an outer loop over the surviving modular values modulo $P_1$, and an inner loop over the surviving modular values modulo $P_2$.  At each iteration of the inner loop, the value of the candidate $a$, bounded by $P_1P_2$, is constructed by the Chinese remainder theorem, and a candidate bit--vector (initially set to zero) is bitwise--ORed with the bit--vector representing whether $a+kP_1P_2$ is the negative of a non--zero quadratic residue. In this manner we test $32$ possible discriminants at each iteration, testing $32P_1P_2>9.8\cdot10^{18}$ possible discriminants once the outer loop has completed.

We check after multiple ORs whether the candidate vector is entirely composed of $1$'s, and repeat with the remaining primes. If a vector survives sieving with all primes up to the 169'th, then the explicit values of the surviving discriminants are computed (i.e. $-(a+kP_1P_2)$, where the $k$'th bit is still zero), and the remainder of Swift's tests are performed upon it. A tally is kept of those discriminants which survive all testing, and output as they arise -- all discriminants were eliminated in the sieving process.
\end{proof}

The code used is available on request.

\section{The main equality}

This follows from \cite{Baker}. Let $k>0$, and $d<0$ be the discriminants of the fields $\mathbb{Q}(\sqrt{k})$, $\mathbb{Q}(\sqrt{d})$ respectively, and $\chi$, $\chi'$ be the Kronecker symbols $\left(\frac{k}{\cdot}\right)$, $\left(\frac{d}{\cdot}\right)$ respectively. Then assuming $(k,d)=1$, we have the following,
\begin{equation*}
 L(1,\chi)L(1,\chi\chi')=\frac{1}{2}\sum_f\left(\frac{\pi^2}{3}\frac{\chi(a)}{a}\prod_{p|k}\left(1-\frac{1}{p^2}\right)+\sum_{r=-\infty}^\infty A_{r,f}e^{\pi irb/(ka)}\right),
\end{equation*}
where the sum over $f$ runs over the set of inequivalent reduced binary quadratic forms of discriminant $d$, and where $a$ and $b$ are the first and second coefficients of $f$. The following hold,
\begin{equation*}
 \lvert A_{r,f}\rvert\leq \frac{4\pi}{\sqrt{\lvert d\rvert}}\lvert r\rvert e^{-\pi\lvert r\rvert\sqrt{\lvert d\rvert}/(ka)}
\end{equation*}
for $r\neq 0$, and 
\begin{equation*}
 A_{0,f} = \begin{cases}
        -\frac{4\pi}{k\sqrt{\lvert d\rvert}}\chi(a)\log p & \text{$k$ a prime power,}\\
        0 & \text{otherwise.}
       \end{cases}
\end{equation*}
By the class number formula,
\begin{align*}
 L(1,\chi) = \frac{2h(k)\log\epsilon}{\sqrt{k}}, && L(1,\chi\chi')=\frac{h(kd)\pi}{\sqrt{k\lvert d\rvert}},
\end{align*}
where $\epsilon$ is the fundamental unit of $\mathbb{Q}(\sqrt{k})$, and $h(k)$, $h(kd)$ are the class numbers of the fields $\mathbb{Q}(\sqrt{k})$ and $\mathbb{Q}(\sqrt{kd})$ respectively.

\section{Estimates for some quantities}

We have previously computed that there are no further discriminants with one class per genus such that $\lvert d\rvert \leq 9.8\cdot10^{18}$ -- we use this fact to simplify the expressions for subsequent bounds. Regarding the forms, for a reduced form $(a,b,a)$ to occur, $\lvert d\rvert=4a^2-b^2$, with $0\leq b\leq a$, so $a\geq\sqrt{\lvert d\rvert/4}$, and $(2a-b)| d$, so that $d$ has a factor between $\sqrt{\lvert d\rvert/4}$ and $\sqrt{\lvert d\rvert}$. However  
$P>2\lvert d\rvert^{1/2}$, so all factors of $\lvert d\rvert$ are either $<\sqrt{\lvert d\rvert}/2$ or $>2\sqrt{\lvert d\rvert}$. So all forms are of the form $(a,0,c)$ or $(a,a,c)$, and so all occurring minima divide $d$.

As $\lvert d\rvert/P\leq\sqrt{\lvert d\rvert}$, applying Robin's bound (Theorem 12, \cite{Robinomega}) on the number of prime factors of a number,
 \begin{equation*}
  \omega(n)\leq\frac{\log n}{\log\log n}+1.45743\frac{\log n}{(\log\log n)^2},
 \end{equation*}
we have, for $\lvert d\rvert\geq 9.8\cdot10^{18}$, as $\lvert d\rvert/P\leq\sqrt{\lvert d\rvert}$,
\begin{equation*}
\omega(d)\leq\frac{\log \sqrt{\lvert d\rvert}}{\log\log \sqrt{\lvert d\rvert}}+1.45743\frac{\log \sqrt{\lvert d\rvert}}{(\log\log \sqrt{\lvert d\rvert})^2}+1\leq\frac{\log\lvert d\rvert}{\log\log\lvert d\rvert}.
\end{equation*}

\subsection{The auxiliary factors}

Here we introduce the parameter $k$ which will be of importance to our analysis, 
\begin{lemma}
 Let $k=q_1q_2$, where $q_1,q_2$ are the first two primes which do not divide $d$, and satisfy $q_1q_2\equiv1(4)$. Then
 \begin{enumerate}\label{kbounds}
  \item $k\leq 1.62(\log\lvert d\rvert)^2$,
  \item The class number of $\mathbb{Q}(\sqrt{k})$ is bounded by $0.64\log\lvert d\rvert$,
  \item The regulator of $\mathbb{Q}(\sqrt{k})$ is bounded by $1.69\log\lvert d\rvert\log\log\lvert d\rvert$,
  \item The class number of $\mathbb{Q}(\sqrt{kd})$ is bounded by $1.01\sqrt{\lvert d\rvert}\log\lvert d\rvert$,
  \item The height of $Q=\prod_{p|k}\left(1-\frac{1}{p^2}\right)$ is bounded by $2.16(\log\lvert d\rvert)^4$.
 \end{enumerate}

\end{lemma}
\begin{proof}

\emph{1}. Firstly, we have $q_1,q_2\leq p_{\omega(d)+4}$, as taking the first three odd primes not dividing $d$, two of them will be equal modulo $4$, so that their product is congruent to $1$ modulo $4$, and hence $k=q_1q_2$ is the discriminant of the quadratic field $\mathbb{Q}(\sqrt{k})$. By results of Rosser and Schoenfeld (see Corollary to Theorem 3, \cite{Rosser}), for $n\geq6$,
\begin{equation*}
  p_n \leq n(\log n+\log\log n),
 \end{equation*}
 where $p_n$ is the $n$'th prime. So the $q_i$ are bounded by
 \begin{equation*}
  q_i\leq p_{\omega(d)+4}\leq(\omega(d)+4)(\log(\omega(d)+4)+\log\log(\omega(d)+4))\leq 1.27\log\lvert d\rvert,
\end{equation*}
and so $k$ is bounded by $k\leq 1.62(\log\lvert d\rvert)^2$.

\emph{2}. By a result of Le (Theorem (a), \cite{Realclassbound}), as $k$ is square--free and $\equiv1(4)$, The class number of the real quadratic field $\mathbb{Q}(\sqrt{k})$ is bounded by $\sqrt{k}/2\leq0.64\log\lvert d\rvert$. 

\emph{3}. Let $\epsilon$ be the fundamental unit of $\mathbb{Q}(\sqrt{k})$, then by a result of Hua, \cite{Huaregulator}, p.329,
\begin{equation*}
 \log\epsilon \leq \sqrt{k}\left(\frac{1}{2}\log(k)+1\right),
\end{equation*}
which yields a bound of
\begin{align*}
 \sqrt{k}\left(\frac{1}{2}\log(k)+1\right)&\leq 1.27\log\lvert d\rvert(\log(1.27\log \lvert d\rvert) +1)\\
 &\leq1.69\log\lvert d\rvert\log\log\lvert d\rvert.
\end{align*}

\emph{4}. By a result of Paulin (Proposition 2.2, \cite{Paulin}), If $D<0$ is an integer congruent to $0$ or $1$ modulo $4$, then 
 \begin{equation*}
  h(D)<\frac{1}{\pi}\sqrt{\lvert D\rvert}(2+\log\lvert D\rvert).
 \end{equation*}
As $kd\equiv0$ or $1(4)$,
\begin{align*}
 h(kd)&\leq2\sqrt{k\lvert d\rvert}(2+\log 2k\lvert d\rvert)/\pi\\
 &\leq 2.54\log\lvert d\rvert\sqrt{\lvert d\rvert}(2+\log( 3.24(\log\lvert d\rvert)^2\lvert d\rvert))/\pi\\
 &\leq 1.01\sqrt{\lvert d\rvert}(\log\lvert d\rvert)^2.
\end{align*}

\emph{5}. For the term $Q=\prod_{p|k}\left(1-\frac{1}{p^2}\right)$, expanding this into a fraction gives
\begin{equation*}
 \prod_{p|k}\left(1-\frac{1}{p^2}\right) = \frac{(q_1^2-1)(q_2^2-1)}{q_1^2q_2^2},
\end{equation*}
and we have $H(Q) \leq 2.61(\log\lvert d\rvert)^4$ by our previous bounds on $q_i$.

\end{proof}

\begin{lemma}
 The quantity 
 \begin{equation*}
  \sum_f\frac{\chi(a)}{a} = \frac{C}{d}
 \end{equation*}
is bounded in height by $1.84\lvert d\rvert\log\log\lvert d\rvert$.
\end{lemma}
\begin{proof}
There are no forms $(a,b,a)$, so the sum $\chi(a)/a$ runs over the smaller half of the divisors of $d$, so 
\begin{equation*}
 \sum_f\frac{\chi(a)}{a} = \frac{C}{d}.
\end{equation*}
where $C$ is an integer. By a result of Robin \cite{Robinsigma}, for $n \geq 3$, 
 \begin{equation*}
  \frac{\sigma(n)}{n}\leq e^{\gamma}\log\log n+\frac{0.649}{\log\log n},
 \end{equation*}
so that $\lvert C\rvert \leq \sigma(d)\leq 1.84\lvert d\rvert\log\log \lvert d\rvert$.

\end{proof}

\subsection{Bound on the remainder}

The minima $a$ are bounded by $\lvert d\rvert/P$, and combined with the previous bound on $k$, this gives
\begin{align*}
 \lvert A_r\rvert &\leq \frac{4\pi}{\sqrt{d}}\lvert r\rvert e^{-\pi\lvert r\rvert\sqrt{\lvert d\rvert}/(ka)}\\
 &\leq\frac{4\pi}{\sqrt{\lvert d\rvert}}\lvert r\rvert e^{-\pi\lvert r\rvert P/(\sqrt{\lvert d\rvert}1.24(\log\lvert d\rvert)^2)},
\end{align*}
and so we have
\begin{align*}
 \left\lvert\sum_{f}\sum_{\substack{r=-\infty\\r\neq0}}^\infty A_re^{\pi irb/(ka)}\right\rvert&\leq\sum_{\substack{r=-\infty\\r\neq0}}^\infty\sum_{f} \left\lvert A_r\right\rvert\\
 &\leq \frac{4\pi h(d)}{\sqrt{\lvert d\rvert}}\sum_{\substack{r=-\infty\\r\neq0}}^\infty \lvert r\rvert e^{-\pi\lvert r\rvert P/(\sqrt{\lvert d\rvert}1.69(\log\lvert d\rvert)^2)}\\
 &=\frac{8\pi h(d)}{\sqrt{\lvert d\rvert}}e^{-\pi P/(\sqrt{\lvert d\rvert}1.69(\log\lvert d\rvert)^2)}\\
 &\quad\quad\quad\cdot(1-e^{-\pi P/(\sqrt{\lvert d\rvert}1.69(\log\lvert d\rvert)^2)})^{-2}\\
 &\leq \frac{27.5 h(d)}{\sqrt{\lvert d\rvert}}e^{-\pi P/(\sqrt{\lvert d\rvert}1.69(\log\lvert d\rvert)^2)},
\end{align*}
where the last line follows as $P\geq1.69\sqrt{d}(\log\lvert d\rvert)^2$, which yields the bound
\begin{equation*}
 (1-e^{-\pi P/(\sqrt{\lvert d\rvert}1.24(\log\lvert d\rvert)^2})\geq(1-e^{-\pi}).
\end{equation*}

\section{The linear form in logarithms}

Putting this together, we have
\begin{equation*}
 \frac{2h(k)\log\epsilon}{\sqrt{k}}\cdot\frac{h(kd)\pi}{\sqrt{k\lvert d\rvert}}=\frac{\pi^2 CQ}{3d}+\sum_{f}\sum_{\substack{r=-\infty\\r\neq0}}^\infty A_re^{\pi irb/(ka)},
\end{equation*}
which yields a linear form in logarithms,
\begin{align*}
  \left\lvert \frac{\sqrt{ \lvert d\rvert}6h(k)h(kd)}{kCQ}\log\epsilon+\pi\right\rvert&= \left\lvert \frac{\sqrt{\lvert d\rvert}6h(k)h(kd)}{kCQ}\log\epsilon-2i\log i\right\rvert\\
  &=\left\lvert \frac{\sqrt{-\lvert d\rvert}3h(k)h(kd)}{kCQ}\log\epsilon+\log i\right\rvert\\
  &\leq\left\lvert\frac{27.5\cdot 3\sqrt{\lvert d\rvert} h(d)}{\pi CQ}e^{-\pi P/(\sqrt{\lvert d\rvert}1.69(\log\lvert d\rvert)^2)}\right\rvert\\
  &\leq 52.6\lvert d\rvert e^{-\pi  P/(\sqrt{\lvert d\rvert}1.69(\log\lvert d\rvert)^2)},
\end{align*}
where the last line follows from  $h(d)<\sqrt{\lvert d\rvert}$, $\lvert C\rvert\geq1$, $\lvert Q\rvert\geq 1/2$.
We now recollect the bounds on the heights of the appearing terms, letting $Q_1$ and $Q_2$ be the denominator and numerator of $Q$ respectively,
\begin{align*}
 H(k)&\leq1.69(\log\lvert d\rvert)^2,\\
 H(h(k))&\leq0.64\log\lvert d\rvert,\\
 H(\sqrt{-\lvert d\rvert}) &= \sqrt{\lvert d\rvert},\\
 H(h(kd))&\leq1.01\sqrt{\lvert d\rvert}(\log\lvert d\rvert)^2,\\
 H(C)&\leq1.84\lvert d\rvert\log\log\lvert d\rvert,\\
 H(Q_1),H(Q_2)&\leq2.61(\log\lvert d\rvert)^4,\\
 H(\epsilon)&\leq \exp(1.69\log\lvert d\rvert\log\log\lvert d\rvert)=:A_1,\\
 H(i)&=1 =:A_2.
\end{align*}
The height of the coefficient $\beta=\frac{\sqrt{-\lvert d\rvert}3h(k)h(kd)Q_1}{kCQ_2}$ is bounded by
\begin{align*}
 H(\beta) &\leq H(\sqrt{-\lvert d\rvert})H\left(\frac{3h(k)h(kd)Q_1}{kCQ_2}\right)\\
 &\leq \sqrt{\lvert d\rvert}1.69(\log\lvert d\rvert)^21.84\lvert d\rvert\log\log\lvert d\rvert2.61(\log\lvert d\rvert)^4\\
 &\leq8.12\lvert d\rvert^{3/2}(\log \lvert d\rvert)^6\log\log\lvert d\rvert=:B.
\end{align*}
We will apply the following theorem of Waldschmidt and Mignotte to this linear form in two logarithms,

\begin{theorem}[Theorem, \cite{Wald}]
 Let $\beta,\alpha_1,\alpha_2$ be three non-zero algebraic numbers of exact degrees $D_0,D_1,D_2$. Let $D$ be the degree of the field $\mathbb{Q}(\beta,\alpha_1,\alpha_2)$ over $\mathbb{Q}$. For $j=1,2$, let $\log\alpha_j$ be any determination of the logarithm of $\alpha_j$, and let $A_j$ be an upper bound for the height of $\alpha_j$ and $\exp\lvert\log\alpha_j\rvert$; further define
 \begin{align*}
  S_0=D_0+\log B, && S_j=D_j+\log A_j,
 \end{align*}
 and assume 
 \begin{equation*}
  \Lambda = \beta\log\alpha_1-\log\alpha_2\neq0.
 \end{equation*}
Let 
 \begin{align*}
  T&=4+\frac{S_0}{D_0}+\log\left(D^2\cdot\frac{S_1}{D_1}\cdot\frac{S_2}{D_2}\right),
 \end{align*}
then 
 \begin{equation*}
  \lvert \Lambda\rvert>\exp\left(-5\cdot10^{8}D^4\cdot\frac{S_1}{D_1}\cdot\frac{S_2}{D_2}\cdot T^2\right).
 \end{equation*}

\end{theorem}
By the Gelfond-Schneider theorem, $\log\epsilon$ and $\log i$ are linearly independent over $\bar{\mathbb{Q}}$, so $\Lambda\neq0$. Applying the theorem to our linear form in logarithms, we have
\begin{align*}
  D_0=D_1=D_2=2, && D=[\mathbb{Q}(\beta,\epsilon,i):\mathbb{Q}]= 8, && S_0&=2+B,\\
  S_1 = 2+\log A_1, && S_2 = 2, && T &= 5+\frac{\log B}{2}+\log(32S_1).
\end{align*}
Now Waldschmidt's result concludes
\begin{equation*}
 \lvert\Lambda\rvert \geq \exp\left(-5\cdot10^8\cdot8^4(1+\log A_1)\left(5+\frac{\log B}{2}+\log(64+32\log A_1)\right)^2\right).
\end{equation*}
Simplifying and using our lower bound $\lvert d\rvert\geq9.8\cdot10^{18}$, we obtain the lower bound
\begin{equation*}
 \lvert\Lambda\rvert \geq \exp\left(-1.2\cdot10^{16} (\log\lvert d\rvert)^3 \log\log\lvert d\rvert\right).
\end{equation*}
We have the upper bound
\begin{equation*}
 \exp\left(-\pi C(\log\lvert d\rvert)^3/1.24+\log\lvert d\rvert+\log50.4\right),
\end{equation*}
and combining it with our lower bound gives
\begin{equation*}
 C(\log\lvert d\rvert)^3\log\log\lvert d\rvert\leq4.8\cdot10^{15}(\log\lvert d\rvert)^3\log\log\lvert d\rvert+1.24\log51.6+1.24\log\lvert d\rvert,
\end{equation*}
so that taking $C=5\cdot10^{15}$ violates this inequality, yielding the theorem.

\end{document}